\documentclass[11pt]{article}
\usepackage{amsfonts}
\usepackage{amsmath}
\usepackage{mathrsfs}
\usepackage{amssymb}
\usepackage{amsmath,color}
\usepackage{amsthm}
\usepackage{amstext}
\usepackage{amsopn}
\usepackage{texdraw}
\usepackage{graphicx}
\usepackage{hyperref}

\usepackage{amsmath}
\usepackage{amsthm}
\usepackage{amstext}
\usepackage{amsopn}
\usepackage{graphicx}
\usepackage{subfigure}
\usepackage{cite}

\input xy
\xyoption{all}

\newtheorem{theorem}{Theorem}[section]
\newtheorem{lemma}{Lemma}[section]

\theoremstyle{definition}
\newtheorem{definition}{Definition}[section]

\theoremstyle{remark}

\numberwithin{equation}{section}

\newcommand{\ba}{\begin{array}}
\newcommand{\ea}{\end{array}}

\newcommand{\ds}{\displaystyle}
\begin{document}
\date{}

\newcommand{\pe}{\psi}
\def\d{\delta}
\def\ds{\displaystyle}
\def\e{{\epsilon}}
\def\eb{\bar{\eta}}
\def\enorm#1{\|#1\|_2}
\def\Fp{F^\prime}
\def\fishpack{{FISHPACK}}
\def\fortran{{FORTRAN}}
\def\gmres{{GMRES}}
\def\gmresm{{\rm GMRES($m$)}}
\def\Kc{{\cal K}}
\def\norm#1{\|#1\|}
\def\wb{{\bar w}}
\def\zb{{\bar z}}

\newcommand{\Pa}{\partial}


\def\bfE{\mbox{\boldmath$E$}}
\def\bfG{\mbox{\boldmath$G$}}

\title {\LARGE \bf Strong Allee Effect in A Stochastic Logistic Model with Mate Limitation and Stochastic Immigration}
\author{Chuang Xu\thanks{Corresponding author. Email: cx1@ualberta.ca.
Department of Mathematics,
Harbin Institute of Technology, Harbin, Heilongjiang 150001,
P. R. China and
Department of Mathematical and Statistical Sciences,
University of Alberta, Edmonton, Alberta,
T6G 2G1 Canada.}}
\maketitle

{\bf Abstract.}
We propose a stochastic logistic model with mate limitation and stochastic immigration. Incorporating stochastic immigration into a continuous time Markov chain model, we derive and analyze the associated master equation. By a standard result,  there exists a unique globally stable positive stationary distribution. We show that such stationary distribution admits a bimodal profile which implies that a strong Allee effect exists in the stochastic model. Such strong Allee effect disappears and threshold phenomenon emerges as the total population size goes to infinity. Stochasticity vanishes  and the model becomes deterministic as the total population size goes to infinity. This implies that there is only one possible fate (either to die out or survive) for a species constrained to a specific community and whether population eventually goes extinct or persists does not depend on initial population density but on a critical inherent constant determined by birth, death and mate limitation. Such a conclusion interprets differently from the classical ordinary differential equation model and thus a paradox on strong Allee effect occurs. Such paradox illustrates the diffusion theory's dilemma.



\section{Introduction}
\label{intro}
\textit{Strong Allee effect} is a phenomenon that the population goes extinct (or grows to the carrying capacity), when it is below (or above) a critical population density, i.e., there exists a minimum density for a single species to persist \cite{D89}. There are two main subcategories for Allee effects- \textit{component Allee effect}, the positive relationship between any measurable component of individual fitness and population density, and \textit{demographic Allee effect}, the positive relationship between the overall individual fitness and population density. Many mechanisms can cause Allee effects: ecological mechanisms, for example, mate limitation, cooperative defense, cooperative feeding, environmental conditioning; genetics mechanisms, for instance, inbreeding depression, and demographic stochasticity, related to the random events of birth rates, death rates, sex ratio, and dispersal \cite{A1,A2,AB,D89,DK,F,G,H,KDLD,L,SSF,W}.

There are lots of studies on models with Allee effects, including both deterministic and stochastic models \cite{AAC,A,CCG,F,GS,GHT,KL,LK,MMA,S,SS,T,VL,WSRW,ZLW}. A logistic model with strong Allee effect generally admits the following dynamics: there exists an unstable equilibrium (the critical density), below which the solution converges to the stable trivial equilibrium, while above which it converges to the stable carrying capacity equilibrium. The critical density determines the minimal number for a species to maintain itself in nature \cite{A2}.

However, when the population size is small, stochastic forces play such a significant role that the assumption of homogeneous mixing is no longer valid. It seems more practical to consider a stochastic model. Dennis \cite{D89} incorporated demographic variability alone and environmental variability alone into deterministic single species models of Allee effects with ordinary density dependence and harvesting. Later, he considered the effect of the combination of the two stochastic forces \cite{D02}. Both discrete birth(-death) processes (modelled by a \textit{master equation}, or \textit{Kolmogorov forward equation}) and their continuous approximations-the diffusion processes (modelled by a stochastic differential equation, or its associated \textit{Fokker-Planck equation}) were used to analyze the stochastic version of Allee effects.

Schreiber \cite{SS} investigated a discrete time single species model with an Allee effect due to predator satiation and mate limitation. Ackleh, Allen and Carter  \cite{AAC} studied a multi-patch stochastic population model with an Allee effect. However, there seems almost no paper addressing  \textit{stochastic (strong) Allee effects}, i.e., there are both positive probabilities for a population of finite size to go extinct and persist even in a single species model. For a formal definition of stochastic strong Allee effect, we refer the reader to Section 4.

It is generally believed that demographic stochasticity enhances the chance of population extinction \cite{L}. For the classical growth model in term of a continuous time Markov chain (CTMC), the extinction state is absorbing; in other words, the population eventually goes extinct with probability one, though the mean time to extinction (MTE) can be exponentially long \cite{D89,NG}. In such cases, there is no chance for the species to survive, and thus Allee effects disappear. However, there is plenty of empirical evidence of thresholds for Allee effects, namely there should be threshold abundances necessary for a species' maintenance \cite{B,D02,GSCR}. Although diffusion process approximation has been commonly used to  depict Allee effects in stochastic models \cite{D89,D02,L}, as  Ovaskainen and  Meerson \cite{OM} pointed out, it remains a challenge to understand population extinction, if the random growth process with discrete states is properly described by a master equation.

Then questions follow: Is a species doomed to extinction irrespective of the initial population size when stochasticity is taken into consideration in CTMC models? Can strong Allee effect exist in stochastic models in term of CTMC? If it could exist, then how to describe it? Are there still notions of  \lq carrying capacity\rq\ and \lq critical density\rq\ as in deterministic models? Which state is it more likely for the chain to eventually stay in, the extinction state, or the carrying capacity state? Is this result consistent with the strong Allee effect demonstrated in deterministic models?

In this paper, we tentatively address the above questions by considering a logistic growth model with mate limitation and stochastic immigration (i.e., spontaneous cases where individuals outside the community move in). With the incorporation of stochastic immigration of species into the population,  there is a positive chance for the species to survive eventually (instead of dying out with probability one because of the absorbing state of the classical CTMC). By a standard result in
\cite{Kam}, there exists a unique  globally asymptotically stable positive stationary distribution (PSD) of the associated master equation. We prove a \textit{stochastic bistability} result: such PSD exhibits a bimodal profile with one peak at the extinction state and another around the deterministic carrying capacity for finite population size. Equivalently, there are both positive probabilities for the population to go extinct and persist, respectively. In this way, we discover the stochastic strong Allee effect in this stochastic model. Further, we prove that such stochastic strong Allee effect disappears and a threshold result holds as the total population size goes to infinity: there exists a critical value, below which the population goes extinct with probability one, and above which persists at the carrying capacity level with probability one. This result establishes a connection between the stochastic logistic model and its deterministic counterpart.  In the deterministic model, strong Allee effect exists irrespective of the total population size: whether the population eventually goes extinct or persists only depends on the initial population density. However, from the perspective of probability, strong Allee effect vanishes if homogeneity is assumed when the population size of the stochastic model goes to infinity: whether the population eventually goes extinct or persists depends not on the initial population density but on a critical constant (related to birth, death and mate limitation) independent of the total population size and the stochastic immigration. Thus a paradox on strong Allee effect occurs. Such a paradox illustrates the \lq\lq diffusion theory's dilemma\rq\rq\ (simply speaking, the diffusion process fails to provide a good global approximation for (nonlinear) stochastic dynamics in term of a master equation as the system size goes to infinity) discovered in some biochemical systems with multiple (stable) equilibria, say the Keizer's paradox \cite{Q,VQ}. In our stochastic model, there is only one equilibrium (either the extinction state or the carrying capacity state) which is globally attractive with full probability as the population size $N\to\infty.$ Whereas in a deterministic model described by a one dimensional ordinary differential equation (ODE) (or a Fokker-Plank equation (FPE) associated with a stochastic differential equation (SDE)), there always exist two locally attractive equilibria (instead of only one globally attractive equilibrium), which is the bistable feature of the strong Allee effect. Thus the corresponding deterministic model with a large population size cannot well approximate the stochastic model in term of a CTMC. For more detailed explanation of the diffusion theory's dilemma, we refer the reader to \cite{Q}.

We would like to mention that, with respect to the rigorous analysis of logistic models in term of a CTMC, J.R. Chazottes, P. Collet and S. M\'{e}l\'{e}ard recently obtained rigorous estimates
of the quasi-stationary distribution and of the MTE in \cite{CCM}.

We formulate our logistic model with mate limitation and stochastic immigration and derive its master equation in the next section. In Section 3,  we state the existence, uniqueness, global stability as well as an explicit formula for the positive stationary distribution. In Section 4, we present our result on stochastic strong Allee effect for finite population size. In Section 5, we show stochastic strong Allee effect disappears and a limit threshold result of population extinction and persistence exists as the population size goes to infinity.
\section{Model formulation}
\label{sec:1}
In this section, we derive a logistic model with mate limitation and stochastic immigration. The model can be described as a birth-death process $\{X(t),\ t\geqslant0\}$ with finite state space $\{0,1,\cdots,N\}.$ For state $X=i,$ the birth rate $b(i)$ and death rate $d(i)$ are specified as follows:
\begin{equation}\label{b}
b(i)=\begin{cases}
\lambda i(1-\delta_1i/N)+\alpha_i(N)\cdot(N-i),\ i=0,1,\cdots,N-1,\\
  0,\ i=N,
\end{cases}
\end{equation}
and \begin{equation}\label{d}
d(i)=\mu i\Big(1+\delta_2i/N+\delta_3\theta/(\theta +i/N)\Big),\ i=0,1,\cdots,N.
\end{equation}
Here $N$ is the largest number of population the area could hold (and thus $b(N)$ is set to be zero). Generally it is greater than the carrying capacity--the largest number of population that may finally survive (which may also depend on many other factors, for instance, the food resource).

The first term $\lambda i(1-\delta_1i/N)$ in $b(i)$ accounts for a general density dependent net birth rate  and the second term $\alpha_i(N)\cdot(N-i)$ represents the stochastic immigration. Note that $\alpha_i(N),$ the stochastic immigration rate is a function in both state $i$ and size $N$ of the total population. As the number of individuals increases, the remaining resources for in-coming individuals including, for instance, food and living space, decrease. When $i=N,$ there is no chance for other population to move in because all the resources are occupied. The term $\mu i(1+\delta_2i/N)$ in $d(i)$ represents a general density dependent net death rate.

Here $\lambda,\ \mu$ are per-individual natural birth rate and natural death rate, respectively; $\delta_1$ and $\delta_2$ are nonnegative constants with $\delta_1\le1$ in order for a nonnegative birth rate $b(i)$ regardless of the stochastic immigration; constant $\delta_3$ is positive. The positive constant $\theta\le1$ is the population density at
which the per-individual birth rate is half of what it
would be if coatings were not limiting \cite{D02}. The central assumption of the logistic is that the difference of the per-individual birth and death rate is a linear declining function of $i/N.$ In our model, $(\lambda-\mu)-(\delta_1\lambda+\delta_2\mu)i/N$ corresponds to such difference. Since $\frac{i/N}{\theta+i/N}$ is the probability of mating, $1-\frac{i/N}{\theta+i/N}=\frac{\theta}{\theta+i/N}$ is the probability of not mating.  The general mate limitation term $\mu i(\delta_3\theta)/(\theta +i/N)$ in $d(i)$ then stands for the reduction of reproduction  due to mating shortage \cite{D89}. For $\delta_3=\lambda/\mu,$ such form of mate limitation reduces to the one investigated in \cite{D89} (see (3.28) on p. 506 in \cite{D89}) for a deterministic counterpart. The mate limitation term used in \eqref{d} can also represent harvesting or predation of Holling type II \cite{Br,D89,Hu,LJH,May,P}.

Since the immigration term $\alpha_i(N)\cdot(N-i)$ may vanish as $N\to\infty$ (as assumption (H) indicates below), we should omit it in the deterministic counterpart. In fact, one of the main goals of this paper is to compare the dynamics of the limit of the stochastic model with the deterministic model as $N\to\infty.$ Thus the corresponding deterministic model (without stochastic immigration) is
\begin{equation}
  \label{ode}
  \frac{dX}{dt}=\lambda X\Big(1-\delta_1X/N\Big)-\mu X\Big[1+\delta_2X/N+\delta_3\theta/(\theta+X/N)\Big].
\end{equation}

Dividing $N$ on both sides of \eqref{ode}, we obtain the ordinary differential equation for the survival rate $x:=X/N:$

\begin{equation}
  \label{ode2}
  \frac{dx}{dt}=\lambda x\Big(1-\delta_1x\Big)-\mu x\Big[1+\delta_2x+\delta_3\theta/(\theta+x)\Big].
\end{equation}

Note that equation \eqref{ode2} is independent of $N.$
We assume \eqref{ode2} admits two positive equilibria $x_-^*<x_+^*$ so that strong Allee effect exists in \eqref{ode}, i.e.,

\vskip 0.3cm
\noindent
$(\mathbf{A}1)$ \ \  $1+\theta(R_0\delta_1+\delta_2)<R_0<1+\delta_3,\ \delta_1^2+\delta_2^2>0$ and the discriminant $\Delta:=\Big[1-R_0-\theta(R_0\delta_1+\delta_2)\Big]^2-4\theta\delta_3(\delta_1R_0+\delta_2)>0$
where \begin{equation}\label{R}
  R_0:=\frac{\lambda}{\mu}
\end{equation} is the \textit{basic reproduction ratio} \cite{Na} and \begin{equation}
  \label{+-}
  x_{\pm}^*=\frac{1}{2(R_0\delta_1+\delta_2)}\Big[R_0-1-(R_0\delta_1+\delta_2)\theta\pm\sqrt{\Delta}\Big]
\end{equation}

In order for the interval $[0,N]$ to be positively invariant under the flow of \eqref{ode}, we further assume
the carrying capacity equilibrium
\vskip 0.3cm
\noindent
$(\mathbf{A}2)$ \ \  $x_+^*\leqslant1.$
\vskip 0.3cm

Moreover, as for the stochastic immigration rates, we assume that:

\vskip 0.3cm
\noindent
$(\mathbf{H})$ \ \  $0\leqslant\alpha_i(N)\leqslant\frac {\mu}N, \ i=0,1,2,\cdots,N-1$.

\vskip 0.3cm
Since $\alpha_i(N)=O(\frac{1}{N}),$ the stochastic immigration noise is negligible as the population size $N\to\infty$. We can see that the extinction state $X=0$ is no longer absorbing as in the classical formulation of the logistic model, if $\alpha_0(N)> 0.$ We remark that when $\alpha_i(N)=0$ for $i=0,1,2,\cdots, N-1$ and $\delta_3=0,$ our model reduces to the general Verhulst logistic CTMC model \cite{Na}. When $\alpha_i(N)=\delta_1=\delta_2=\delta_3\theta=0,$ $i=0,\cdots,N-1,$ our model then reduces to a closed susceptible to infective to susceptible (SIS) epidemic model \cite{WD}.

For notational convenience, let
\begin{equation}\label{r1j}
R_{1i}=\frac {N\alpha_i(N)}{\mu}, \quad i=0,1,\cdots N.
\end{equation}
Then, assumption $({\rm H})$ is equivalent to $$0\leqslant R_{1i}\leqslant1,\ \text{for}\ i=0,1,2,\cdots,N-1.$$

For $i=0, 1, 2, \cdots, N,$ let
$$
p_i(t)=\text{Prob}\{\, X(t)=i\,\}
$$
be the probability that the population is at the state $X=i$, or equivalently, the
probability that there are $i$ individuals surviving in the total population. Then
$$
p(t)=(p_0(t), p_1(t), \cdots, p_N(t))
$$
is the probability distribution at time $t.$
According to \cite{G,K}, the master equation for our stochastic logistic
model with mate limitation and stochastic immigration can be written as
\begin{equation}\label{1-1}
\dot{p}=Qp,
\end{equation}
with
\begin{equation}\label{q}
Q=\begin{bmatrix}
  -b(0)&d(1)&0&\cdots&0\\
    b(0)&-[b(1)+d(1)]&d(2)&\cdots&0\\
        0&b(1)&-[b(2)+d(2)]&\cdots&0\\
      0&0&b(2)&\cdots&0\\
      \vdots&\vdots&\vdots&\vdots&\vdots\\
      0&0&0&\cdots&d(N)\\
      0&0&0&\cdots&-d(N)\\
\end{bmatrix}.
\end{equation}

Hereafter, master equation \eqref{1-1} will be studied as a system of differential equations in the positively invariant bounded subset of the
phase space $\mathbb{R}^{N+1}_+:$
\begin{equation}\label{gamma}
\Gamma= \Big\{\,(p_0, p_1, \cdots, p_N)\in \mathbb{R}^{N+1}_+ \quad \Big| \quad \sum_{i=0}^N p_i=1\, \Big\}.
\end{equation}

For the deterministic model \eqref{ode}, it is known that $0$ and $x_+^*N$ are two stable nodes, and $x_-^*N$ is an unstable node. In the feasible region $[0,N],$ trajectories of \eqref{ode} converge to $0$ if initiated below $x_-^*N$ and converge to $x_+^*N$  if initiated above $x_-^*N.$ Such bistable dynamics correspond to the strong Allee effect. The following sections target at investigating the stochastic strong Allee effect in our stochastic logistic model as we mention in the Introduction.
\section{Positive stationary distribution (PSD)}
\label{sec:2}
Let  $A=(a_{ij})_{(N+1)\times (N+1)}$ be the nonnegative matrix with entries
\begin{equation}\label{0-3}
a_{ij}=\begin{cases}
  \lambda j(1-\delta_1j/N)+\alpha_j(N)\cdot(N-j), \quad  &i=j+1,\\
  \mu j[1+\delta_2j/N+(\delta_3\theta N)/(\theta N+j)], \quad   &i=j-1,\\
  0, \quad &i\neq j\pm1.
\end{cases}
\end{equation}
Then $Q=-L(A)$, where $L(A)={\rm diag}(d_1,\cdots, d_n)-A$
is the algebraic Laplacian matrix of $A$ with $d_j$ representing the sum of the $j$-th column of $A$ \cite{W}.
A \textit{stationary distribution} $p^s\in\Gamma$ is an equilibrium of \eqref{1-1}.

Now we claim the standard results on the existence, uniqueness and global stability of a positive stationary distribution of \eqref{1-1} \cite{Kam}.

\vskip 0.3cm
\begin{theorem}[\textbf{Existence, Uniqueness and Global Stability of PSD}]\label{th0}
Suppose that $\alpha_0(N)>0.$ Then master equation \eqref{1-1} has a unique positive stationary distribution $p^s=(p_0^s,p_1^s,\cdots,p_N^s)$ and $p^s$ is globally asymptotically stable in $\Gamma.$
\end{theorem}

Next, we present a formula for the positive stationary distribution $p^s.$ For a general formula for stationary distribution of a birth-death process, we refer the reader to p.12 in \cite{GR}.

\vskip 0.3cm
\begin{theorem}[\textbf{Formula for PSD}]\label{th10} Assume that $\alpha_0(N)>0$. Let $p^s=(p_0^s, p_1^s, \cdots, p_N^s)$ be
the unique positive stationary distribution of \eqref{1-1}. Then
 \begin{equation}\label{3-2}
 p_i^s=p_0^s\prod_{j=0}^{i-1}\frac{R_0j(1-\delta_1j/N)+R_{1j}(N-j)/N}{(j+1)\Big[1+\frac{\delta_2(j+1)}{N}+\frac{\delta_3\theta}{\theta+(j+1)/N}\Big]}, \quad i=1,\cdots, N,
 \end{equation}
with
\begin{equation}\label{3-3}
 p_0^s=\Big(1+\sum_{i=1}^N\prod_{j=0}^{i-1}\frac{R_0j(1-\delta_1j/N)+R_{1j}(N-j)/N}{(j+1)\Big[1+\frac{\delta_2(j+1)}{N}+\frac{\delta_3\theta}{\theta+(j+1)/N}\Big]}\Big)^{-1};\end{equation}
where $R_0$ and $R_{1j}$ ($j=0,\cdots,N-1$) are defined in \eqref{R} and \eqref{r1j}, respectively.
\end{theorem}

Although from Theorem~\ref{th0} that the asymptotic dynamics does not depend on the initial probability distribution, which means that strong Allee effect in the deterministic sense disappears, it still makes sense to investigate whether there are both positive probabilities that the population goes extinct and persists respectively, which can be identified as a stochastic version of strong Allee effect for a CTMC growth model.
\section{Stochastic strong Allee effect for finite population}
\label{sec:3}
In this section, we investigate  strong Allee effect in the stochastic model \eqref{1-1} via the profile of the PSD. The population is \textit{extinct in probability} if the profile peaks at some $i_1=i_1(N)$ such that $\lim_{N\to\infty}\frac{i_1}{N}=0.$ The population is \textit{persistent in probability} if the profile peaks at some $i_2=i_2(N)$ such that $\lim_{N\to\infty}\frac{i_2}{N}>0.$ One should notice that if $\lim_{N\to\infty}\frac{i_1}{N}$ (or $\lim_{N\to\infty}\frac{i_1}{N},$ respectively) does not exist, then it makes nonsense to talk about extinction in probability (or persistence in probability, respectively) according to our definition. If the population is both extinct in probability and persistent in probability, then model \eqref{1-1} is said to admit a \textit{stochastic strong Allee effect}.

We first investigate the profile of PSD of model \eqref{1-1} and then show that stochastic  strong Allee effect exists in model \eqref{1-1} for finite population size.
\vskip 0.2cm
\begin{theorem}[\textbf{Profile of PSD}]\label{th1} Assume that $\alpha_0(N)>0$. Let $p^s=(p_0^s,p_1^s,\cdots,p_N^s)$
be the unique positive stationary distribution of \eqref{1-1}. Then $p^s$ is bimodal. Specifically,
there exist $0\leqslant i_-<i_+\leqslant N$ and $m\in\mathbb{N}^+$ independent of $N$ such that for sufficiently large $N$,
  $p_i^s$ is decreasing for $0\le i\le i_--m$ and $i_++m\le i\le N$ while increasing for $i_-+m\le i\le i_+-m.$ Moreover, $p_{i_+}^s=\underset{i_+-m\leqslant i\leqslant i_++m}{\max}\{p_i^s\}$ and $p_{i_-}^s=\underset{i_--m\leqslant i\leqslant i_-+m}{\min}\{p_i^s\};$ $i_{\pm}$ have the following asymptotic expansions
 \begin{equation}\label{imax}
 \frac{i_{\pm}}{N}=x_{\pm}^*+O\Big(\frac{1}{N}\Big), \quad \text{as}\quad\ N\to\infty,
 \end{equation}
 where $x_{\pm}^*$ are defined in \eqref{+-}.
\end{theorem}
\vskip 0.2cm
\begin{proof}
  Using \eqref{3-2}, we have
\begin{equation}\label{ppi}
      \frac{p_{i+1}^s}{p_i^s}=\displaystyle\frac{R_0i(1-\delta_1i/N)+R_{1i}(N-i)/N}{(i+1)\Big[1+\frac{\delta_2(i+1)}{N}+\frac{\delta_3\theta}{\theta+(i+1)/N}\Big]}.
\end{equation}

From \eqref{ppi}, to find $i_{\pm}$, we look for
$i$ such that $p^s_{i-1}\le p_i^s$ and $p_i^s\ge p_{i+1}^s.$ It suffices to consider $i$ for $p^s_{i+1}=p^s_i.$  Note that the cubic equation in $i$
\begin{equation*}
  \frac{R_0i(1-\delta_1i/N)+R_{1i}(N-i)/N}{(i+1)\Big[1+\frac{\delta_2(i+1)}{N}+\frac{\delta_3\theta}{\theta+(i+1)/N}\Big]}=1
\end{equation*}
has two positive roots $r_-^{(i)}<r_+^{(i)}$ and one negative root $r_0^{(i)}.$

In fact, let $x^{(i)}=i/N,$ then $x^{(i)}$ solves the cubic equation
\begin{equation}
  \label{a0}
  \begin{split}
  h^{(i)}_N(x)=&-(R_0\delta_1+\delta_2)x^3-\{[1-R_0+\theta(R_0\delta_1+\delta_2)]+\frac{3\delta_2+R_0\delta_1+R_{1i}}{N}\}x^2\\&-\Big[\theta(1+\delta_3-R_0)+\frac{2\delta_2\theta+2-R_0+R_{1i}\theta-R_{1i}}{N}+\frac{3\delta_2+R_{1i}}{N^2}\Big]x\\&-\Big[\frac{(\delta_3+1-R_{1i})\theta}{N}+\frac{1+\delta_2\theta-R_{1i}}{N^2}+\frac{\delta_2}{N^3}\Big]=0.
\end{split}\end{equation}
 Let $N\to \infty,$ by assumption ({\rm H}), equation \eqref{a0} becomes
\begin{equation}
  \label{a1}
  h(x)=-(R_0\delta_1+\delta_2)x^3-[1-R_0+\theta(R_0\delta_1+\delta_2)]x^2-\theta(1+\delta_3-R_0)x=0,
  \end{equation}
which admits three roots $0,$ $x_{\pm}^*,$ where $x^*_\pm$ are defined in \eqref{+-}. Since $x_-^*\neq x_+^*,$ for sufficiently large $N,$ all the three roots of \eqref{a0} are real. Note that by assumption (H),
$$h^{(i)}_N(0)=-\Big[\frac{(\delta_3+1-R_{1i})\theta}{N}+\frac{1+\delta_2\theta-R_{1i}}{N^2}+\frac{\delta_2}{N^3}\Big]<0.$$ By the monotonicity of $h(x),$ we know that $r_0^{(i)}<0.$

Next using regular perturbation, we have the asymptotic expansion for $r_0^{(i)},$ $r_-^{(i)}$ and $r_+^{(i)}:$
\begin{equation}
  \label{r0}r_0^{(i)}=-\frac{\delta_3+1-R_{1i}}{\delta_3+1-R_0}+O(\frac{1}{N});
\end{equation}
\begin{equation}
  \label{r+-}\begin{split}r_{\pm}^{(i)}=&x_{\pm}^*N\\&-\frac{(3\delta_2+R_0\delta_1+R_{1i})(x_{\pm}^*)^2+(2\delta_2\theta+2-R_0+R_{1i}\theta-R_{1i})x_{\pm}^*+(\delta_3+1-R_{1i})\theta}{3a(x_{\pm}^*)^2-2bx_{\pm}^*+c}\\&+O(\frac{1}{N}),
  \end{split}
\end{equation}
where $a=R_0\delta_1+\delta_2,$ $b=-\Big[1-R_0+\theta(R_0\delta_1+\delta_2)\Big]$ and $c=\theta(1-\delta_3+R_0).$ Note that $3a(x_{\pm}^*)^2-2bx_{\pm}^*+c\neq0.$ In fact, by assumption (A1), $a,\ b,\ c>0.$ Since $a(x_{\pm}^*)^2-bx_{\pm}^*+c=0,$ we have \[
\begin{split}
  &3a(x_{\pm}^*)^2-2bx_{\pm}^*+c\\
  =&3(bx_{\pm}^*-c)-2bx_{\pm}^*+c\\
  =&bx_{\pm}^*-2c\\
  =&\frac{\pm2c\sqrt{b^2-4ac}}{b\mp\sqrt{b^2-4ac}}\neq0.
\end{split}
\]
Hence from \eqref{r+-}, we have
\begin{equation}\label{2-1}
r_-^{(i)}-r_-^{(j)}=\frac{(R_{1j}-R_{1i})(x_-^*-1)(x_-^*+\theta)}{3a(x_-^*)^2-2bx_-^*+c}+O\Big(\frac{1}{N}\Big)
\end{equation}
and
\begin{equation}\label{2-2}
r_+^{(i)}-r_+^{(j)}=\frac{(R_{1j}-R_{1i})(x_+^*-1)(x_+^*+\theta)}{3a(x_+^*)^2-2bx_+^*+c}+O\Big(\frac{1}{N}\Big).
\end{equation}
Consider sets $A_-=\{\lfloor r_-^{(i)}\rfloor+1\ | \ 0\leqslant i\leqslant N \}$  and
$A_+=\{\lfloor r_+^{(i)}\rfloor+1 \ | \ 0\leqslant i\leqslant N\},$ where $\lfloor x\rfloor$ is the floor function of $x$ (i.e., the largest integer not exceeding $x$). Then by \eqref{2-1}, \eqref{2-2} and assumption ({\rm H}), $A_{\pm}$
have no more than $m$ (independent of $N$) elements. Let $i_+\in A_+$ and $i_-\in A_-$ be such that $$p^s_{i_+}\geqslant p_i^s,\
\text{for all}\ i\in A_+,\ \text{and}\  p^s_{i_-}\leqslant p_i^s,$$
for all $i\in A_-$. Then $p_i^s$ is increasing in $[i_-+m,i_+-m]$
and decreasing in $[0,i_--m]$ and $[i_++m,N]$.
This completes the proof of Theorem~\ref{th1}.
\end{proof}

From Theorem~\ref{th1}, we see that the stochastic strong Allee effect does exist in model \eqref{1-1} for finite population size; nevertheless, it is still not clear whether the population is more likely to go extinct or persist, though there are both positive probabilities that these two events occur.
Results of Theorem~\ref{th1} are illustrated by numerical simulations in Figure~\ref{fig1}.
Bistable dynamics of \eqref{1-1} are clearly captured  for finite
population size $N.$ In (a), the PSD has a major peak at $X=0$ and a minor peak at $X=40;$ while in (b), the PSD has a major peak at $X=35$ and a minor peak at $X=0,$ which matches well with the result observed in \cite{DK} (see Fig. 4 (c) in \cite{DK}). However, when $N$ becomes sufficiently large, such bistability feature seems not apparent as shown in (c) and (d).
\section{A paradox on strong Allee effect}
In the following, we study the stochastic asymptotic dynamics of model \eqref{1-1} as the total population size $N\to\infty.$ As Theorem~\ref{th3} will show, stochastic strong Allee effect disappears as $N\to\infty.$ Stochasticity disappears when the population size is infinity and thus such model should correspond to the deterministic model. However, as the result will demonstrate, there is only one possible fate for the species: to die out or to survive and such fate is not determined by the initial population density but by an inherent constant of the deterministic model. Such biological conclusion is quite different from what the deterministic ODE model \eqref{ode} interprets. In this way, a paradox on strong Allee effect occurs. This gives another example to illustrate the diffusion theory's dilemma as explained in the Introduction.

Beforehand, we first define \textit{Markov exponent}, \textit{stochastic asymptotic extinction} and \textit{stochastic asymptotic persistence} of model \eqref{1-1}.

\begin{definition}
  Let $\{i_N\},\ \{j_N\}\subset\mathbb{N}$ {\rm(}i.e., $\{i_N\}$ and $\{j_N\}$ are two subsequences of the set of natural numbers{\rm)} satisfy $0\le i_N,\ j_N\le N.$ We call $M(i_N,j_N)$ the Markov exponent of $\{i_N\}$ to $\{j_N\}$  if $M(i_N,j_N):=\lim_{N\to\infty}\frac{1}{N}\log\frac{p^s_{i_N}}{p^s_{j_N}}$ exists. Obviously, $M(j_N,i_N)=-M(i_N,j_N)$ if existing. If $M(i_N,j_N)>0,$ then $p^s_{j_N}$ decays exponentially faster than $p^s_{i_N}.$ Markov exponent measures the relative exponential decay rate of one sequence of probabilities to the other.
\end{definition}

\begin{definition}
  Let $\{k_N\}\subset\mathbb{N}$ satisfy $0\le k_N\le N.$
   \begin{enumerate}
\item[{\rm(1)}] The population is \textit{stochastically asymptotically weakly {\rm(}strongly{\rm)} extinct} if  $$\limsup_{N\to\infty}\frac{p^*_{k_N}}{p^*_{i_N}}=\infty\  {\rm(}\liminf_{N\to\infty}\frac{p^*_{k_N}}{p^*_{i_N}}=\infty{\rm)}\ \text{and}\ \lim_{N\to\infty}\frac{k_N}{N}=0,$$ for all $\{i_N\}\subset\mathbb{N}$ satisfying $0\le i_N\le N$ and $\lim_{N\to\infty}\frac{k_N}{i_N}\in(0,1)\cup(1,\infty).$
     \item[{\rm(2)}]  The population is \textit{stochastically asymptotically weakly {\rm(}strongly{\rm)} persistent} if  $$\limsup_{N\to\infty}\frac{p^*_{k_N}}{p^*_{i_N}}>1\ {\rm(}\liminf_{N\to\infty}\frac{p^*_{k_N}}{p^*_{i_N}}>1{\rm)}\ \text{and}\ \lim_{N\to\infty}\frac{k_N}{N}>0,$$ for all $\{i_N\}\subset\mathbb{N}$ satisfying $0\le i_N\le N$ and $\lim_{N\to\infty}\frac{k_N}{i_N}\in(0,1)\cup(1,\infty).$
        \end{enumerate}
\end{definition}
\vskip 0.2cm
For $a\in [0, 1],$ denote the Dirac
delta measure at $a$  by $\delta_a$.
Next we state the following stochastic limit threshold theorem with respect to the stochastic asymptotic extinction and persistence for model \eqref{1-1}.

For a fixed $N,$ a random variable $Y_N$ taking values in $\{i/N\}_{i=0}^N$ is called a \textit{stationary population density} of model \eqref{1-1} if $\mathbb{P}(Y_N=i/N)=p_i^s.$ We use  $p^s_d$ to denote the distribution of stationary population density of model \eqref{1-1}. In the following, we do not distinguish two random variables having the same distribution $p^s_d.$ In other words, the stationary population density is unique in this sense.
\vskip 0.2cm
\begin{theorem}[\textbf{Stochastic Limit Threshold Theorem}]\label{th3}\hfill

Assume $\alpha_0(N)>0.$ Let  \begin{equation}
  \label{f}f(x)=\frac{R_0(1-\delta_1 x)}{1+\delta_2x+\frac{\delta_3\theta}{\theta+x}}.
\end{equation}
\begin{enumerate}
\item[{\rm (a)}] Suppose that $\int_0^{x_+^*}\log f(x)dx<0.$  Then
the stationary population density $Y_N$ converges to $0$ in distribution as $N\to \infty;$ or equivalently the species eventually goes extinct with probability one. In particular, the population is stochastically asymptotically strongly extinct.
\item[{\rm (b)}] Suppose that $\int_0^{x_+^*}\log f(x)dx>0$ and $\lim_{N\to\infty}\frac{1}{N}\log \alpha_0(N)=0,$
then stationary population density $Y_N$ converges to $x_+^*$ in distribution as $N\to \infty;$ or equivalently the species eventually persists at the deterministic carrying capacity level with probability one. In particular, the population is stochastically asymptotically strongly persistent.
\end{enumerate}
\end{theorem}
\begin{proof}
  \begin{enumerate}
\item[(a)]. Suppose $\int_0^{x_+^*}\log f(x)dx<0.$ For small $\varepsilon\in(0,x_-^*)$, by Theorem~\ref{th1}, \eqref{M}, and proof of Lemma~\ref{le1} in the Appendix on the Markov exponent of $i_+$ to $0,$ we have there exists $N_0,$ such that $\forall\ N>N_0,$
    $$p_i^s<e^{N\cdot\frac{1}{2}\max\{\int_0^{\varepsilon}\log f(x)dx,\ \int_0^{x_+^*}\log f(x)dx\}}p_0^s,\ \forall\ i=\lfloor\varepsilon N\rfloor$$
$$
\mathbb{P}[Y_N>\varepsilon]<\sum_{i=\lfloor\varepsilon N\rfloor}^Np_i^s<Ne^{N\cdot\frac{1}{2}\max\{\int_0^{\varepsilon}\log f(x)dx,\ \int_0^{x_+^*}\log f(x)dx\}}p_0^s\rightarrow0,
$$as $N\rightarrow\infty$,
and therefore the stationary population density $Y_N$ converges to $0$ in probability and thus in distribution by Portmanteau theorem. For a detailed statement of Portmanteau lemma, we refer the reader to Lemma~2.2 on  p.6 in \cite{Van}.
\vskip 0.2cm
\item[(b)]. Suppose $\int_0^{x_+^*}\log f(x)dx>0.$ As in (a), we can show that for small $\varepsilon>0,$ there exists $N_0,$ such that $\forall\ N>N_0,$
\[\begin{split}
  &\mathbb{P}[|Y_N-x_+^*|>\varepsilon]\\
  <&Ne^{N\cdot\frac{1}{2}\max\{-\int_0^{x_+^*}\log f(x)dx,\ -\int_{x_+^*-\varepsilon}^{x_+^*}\log f(x)dx,\ \int_{x_+^*}^{x_+^*+\varepsilon}\log f(x)dx\}}p_{i_+}^s\rightarrow0,
\end{split}
\]
$\text{as $N\rightarrow\infty$.}$ This implies that  the stationary population density $Y_N$ converges to $x_+^*$ in distribution.
\end{enumerate}
Now we complete the proof of the Theorem~\ref{th3}.
\end{proof}
\begin{figure}
\centering
    \subfigure[$N=100.$]{\includegraphics[width=2.1in]{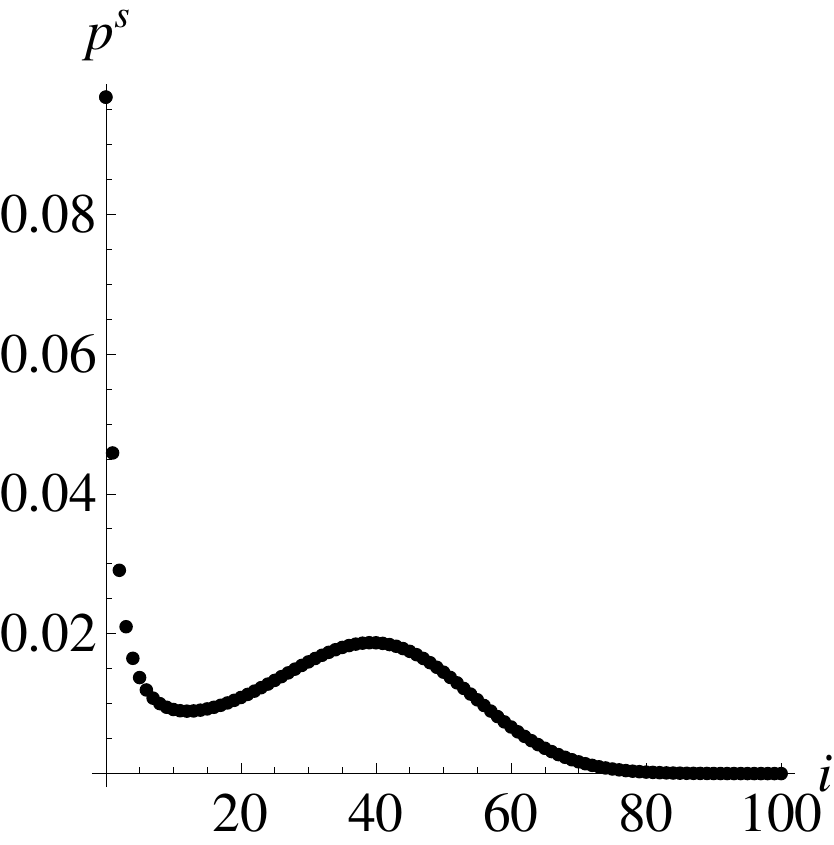}}\hspace{1cm}
    \subfigure[$N=100.$]{\includegraphics[width=2.1in]{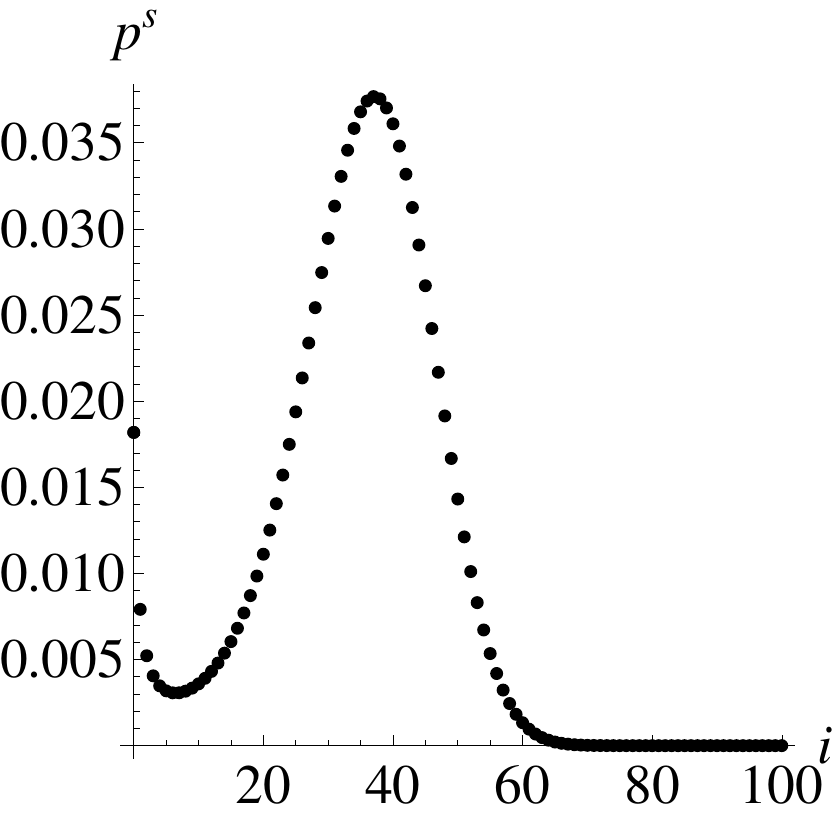}}\\
    \subfigure[$N=500.$]{\includegraphics[width=2.1in]{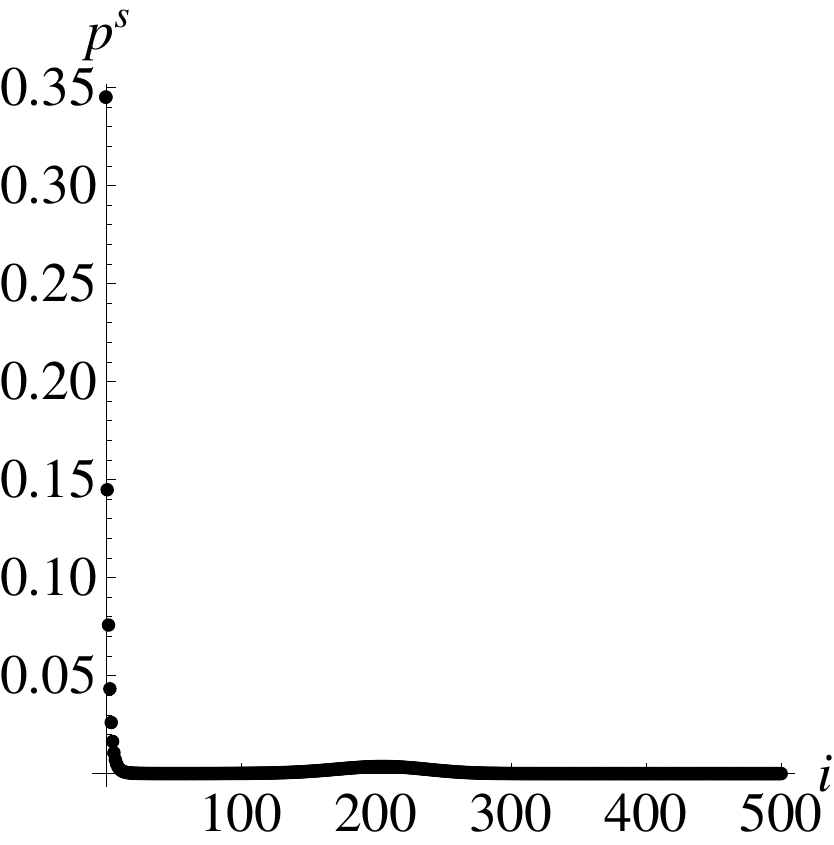}} \hspace{1cm}
    \subfigure[$N=500.$]{\includegraphics[width=2.1in]{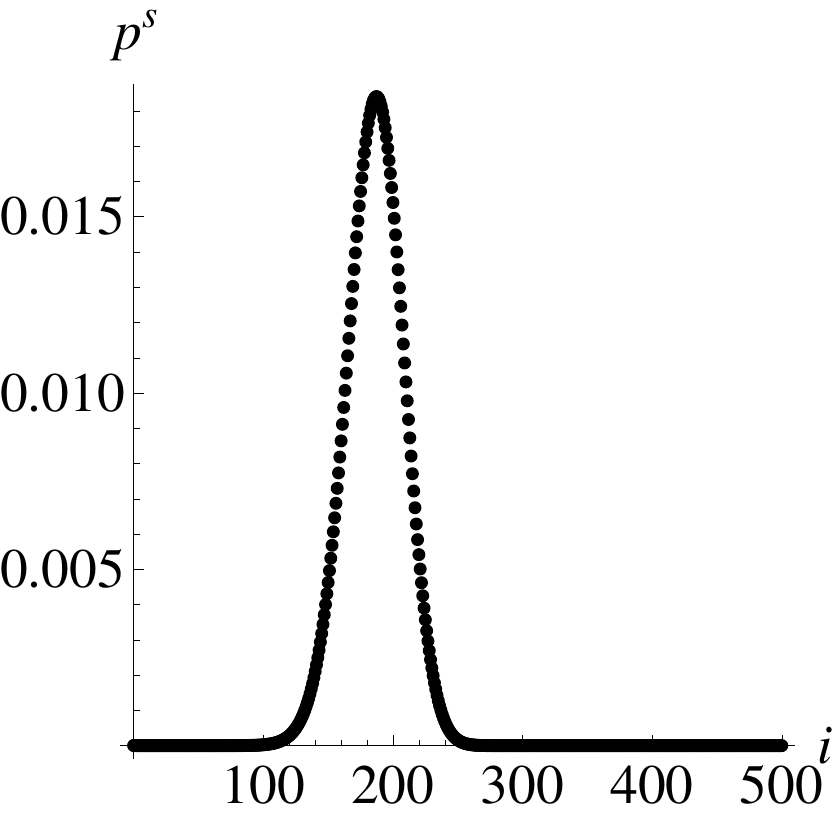}}
   \caption{\small Profiles of the positive stationary distribution $p^s.$ In {\rm (a)} an {\rm (c)}, $R_0=1.4,$ $\delta_1=0.45,$ $\delta_2=0.1,$ $\delta_3=1.45,$ $\theta=0.03,$ $R_{1i}=0.99,$ for $i=0,\cdots,N-1;$ in {\rm (b)} and {\rm (d)}, $R_0=1.7,$ $\delta_1=0.9,$ $\delta_2=0,$ $\delta_3=1.7,$ $\theta=0.03,$ $R_{1i}=0.99,$ for $i=0,\cdots,N-1.$ In {\rm (a)} and {\rm (b)}, bimodal profile of the positive stationary distribution $p^s$ is clearly captured. However, in {\rm (c)} and {\rm (d)}, the distinctive bimodal features of the two profiles become less prominent such that stochastic strong Allee effect seems to disappear as the total population size increases.}\label{fig1}
\end{figure}

\vskip 0.2cm

\begin{figure}
\centering
     \subfigure[$N=5000$ and $\int_0^{y_+^*}\log f(x)dx<0.$]{\includegraphics[width=2.1in]{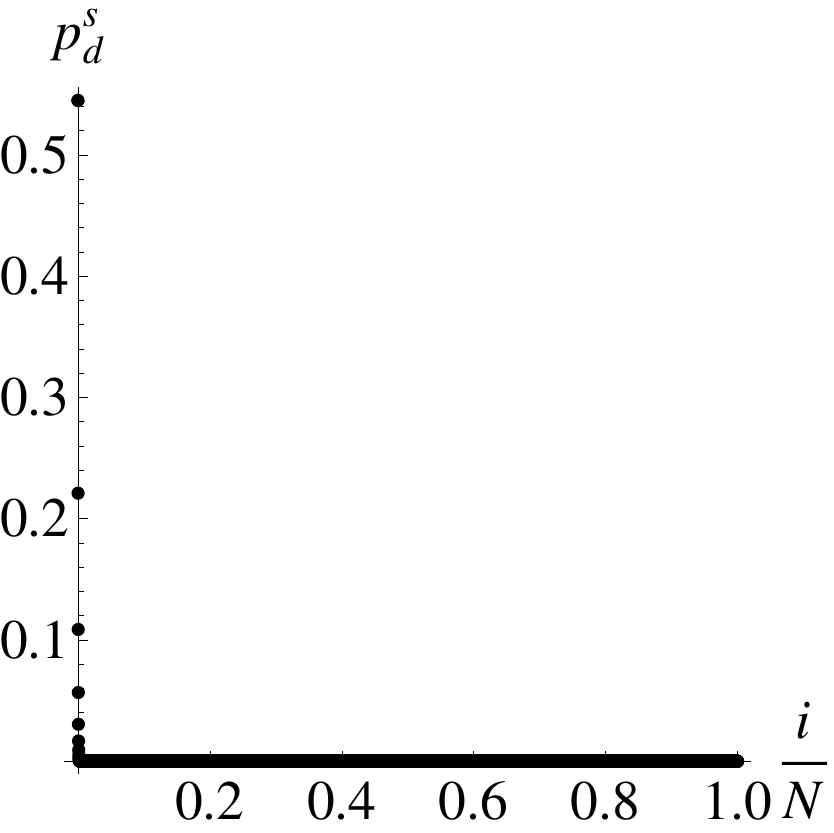}} \hspace{1cm}
    \subfigure[$N=5000$ and $\int_0^{y_+^*}\log f(x)dx>0.$]{\includegraphics[width=2.1in]{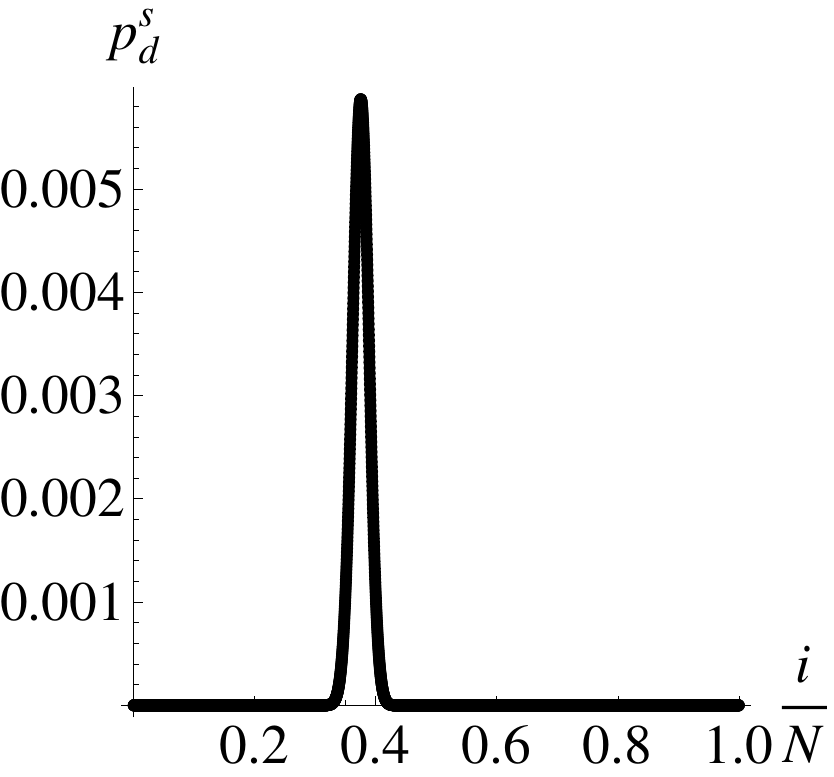}}
\caption{\small Profiles of distribution $p^s_d$ of the stationary population density. In {\rm (a)},  $R_0=1.4,$ $\delta_1=0.45,$ $\delta_2=0.1,$ $\delta_3=1.45,$ $\theta=0.03,$ $R_{1i}=0.99,$ for $i=0,\cdots,N-1;$ $x_+^*=0.413621,$ $x_-^*=0.104324,$ $\int_0^{x_+^*}\log f(x)dx=-0.00611319.$ The probability
density is highly concentrated in a neighbourhood of $0.$ In {\rm (b)}, $R_0=1.7,$ $\delta_1=0.9,$ $\delta_2=0,$ $\delta_3=1.7,$ $\theta=0.03,$ $R_{1i}=0.99,$ for $i=0,\cdots,N-1;$ $x_+^*=0.375266,$ $x_-^*=0.0522505,$ $\int_0^{x_+^*}\log f(x)dx=0.0207001.$ The probability
density is highly concentrated in a neighbourhood of $x_+^*.$} \label{fig2}
\end{figure}

\vskip 0.2cm

Theorem~\ref{th3} states that, as the population size $N\to \infty,$ the distribution of the stationary population density has a limit, which exhibits a sharp threshold result.  Theorem~\ref{th3} also indicates that for large population size $N,$ whether the population is more likely to go extinct or to persist only depends on the critical parameter $\int_0^{x_+^*}\log f(x)dx.$ Such a result shows that a species will either survive with probability one or die out with probability one and whether a species will survive or die out is determined by some inherent constant instead of the initial population density. This interprets differently from the deterministic model \eqref{ode}. Hence a paradox on strong Allee effect occurs.

Nevertheless, it is still possible that as $N\to\infty,$ the discrete measure $\sum_{i=1}^Np_i^s\delta_{i/N}$ associated with the stationary population density converges to a convex combination of the two Dirac delta measures $\delta_0$ and $\delta_{x_+^*}$ which is also singular to the Lebesgue measure when $\int_0^{x_+^*}\log f(x)dx=0.$

Results of Theorem~\ref{th3} are illustrated
in Figure~\ref{fig2}. We comment that for our results to hold,  especially for statement (b) in Theorem~\ref{th3}, stochastic immigration $\alpha_0(N)$ does not have to be very big; it is only
required not to decay exponentially fast as the population size $N\to\infty.$

Note that even if $\alpha_i=0$ for $i\neq0,$ main results in this paper still hold. However if $\alpha_0=0,$ then state $i=0$ is absorbing and the stationary distribution is degenerate, meaning that the population goes extinct with probability one. Mathematically, the seemingly negligible stochastic immigration ($\alpha_0\neq0$) at state $i=0$ that destroys the absorbing state of the Markov chain triggers the stochastic strong Allee effect. Nevertheless, this immigration, though may affect the outcome of the strong Allee effect (as we will illustrate in the following), is not the cause (which is the mate limitation) of the strong Allee effect in the deterministic model.

For large $\alpha_i,$ that is, the immigration is not taken to be noise, the corresponding deterministic counterpart of the Markov chain model is

\begin{equation}\label{ode1}
   \frac{dX}{dt}=\lambda X\Big(1-\delta_1X/N\Big)+\alpha(N-X)-\mu X\Big[1+\delta_2X/N+\delta_3\theta/(\theta+X/N)\Big]
 \end{equation}
and the ODE for $x=X/N$ is
\begin{equation}\label{ode3}
   \frac{dx}{dt}=\lambda x\Big(1-\delta_1x\Big)+\alpha(1-x)-\mu x\Big[1+\delta_2x+\delta_3\theta/(\theta+x)\Big].
 \end{equation}
 Here the immigration rate $\alpha$ can be viewed as the average of $\alpha_i.$ Note for $\alpha=0,$ \eqref{ode3} admits three equilibria $0<x_-^*<x_+^*$ with $0$ and $x_+^*$ being stable and $x_-^*$ unstable. Also $x$ is an equilibrium of \eqref{ode3} if and only if $$(\lambda\delta_1+\mu\delta_2)x^3+[\theta(\lambda\delta_1+\mu\delta_2)+\alpha+\mu-\lambda)x^2+[\theta(\mu\delta_3+\alpha+\mu-\lambda)-\alpha]x-\alpha\theta=0.$$Thus by continuous dependence, we know there exists $\alpha_0>0$ such that $\forall\ 0<\alpha<\alpha_0,$ \eqref{ode3} admits two stable positive equilibria $x_2<x_3$ and one unstable negative equilibrium $x_1.$ In other words, $x_3$ is the unique stable positive equilibrium. Thus for the new stochastic model with large stochastic immigration, one should expect a unimodal PSD with the unique peak around $x_3.$ This indicates that when immigration effect is strong enough, the nonnegative solution of \eqref{ode1} will always converge to the carrying capacity equilibrium and strong Allee effect disappears, for there is always population moving in.

\section{Conclusion}

In this paper, we formulate a stochastic logistic model with mate limitation and stochastic immigration in term of a CTMC and investigate strong Allee effect in this model. For the associated master equation, a unique positive stationary distribution (PSD) exists and is globally asymptotically stable. The PSD has a bimodal profile and thus the population can both go extinct and persist with positive probability for finite population size; in other words, a stochastic strong Allee effect exists in this stochastic model. We further prove that such stochastic strong Allee effect disappears and a threshold result holds as the population size tends to infinity: the population goes extinct with probability one if a critical parameter is below 0, while persists at the carrying capacity with probability one if the critical value is above 0. In other words, there is only one possible fate if the model is described deterministically (when the population size is assumed to be infinity), and whether the population finally survives or dies out (with probability one) does not depend on the initial density but on an inherent parameter determined by birth rate, death rate and mate limitation. Such threshold phenomenon is inconsistent with the classical bistable result for this logistic model \eqref{ode2} with mate limitation and a paradox on strong Allee effect occurs which illustrates the diffusion theory's dilemma.

\section*{Acknowledgements}

The author is indebt to two anonymous referees' valuable comments and suggestions which greatly improve the presentation of this paper. The author would also like to thank Dr. J.R. Chazottes for bringing reference \cite{CCM} to his attention.

\section*{Appendix: Markov Exponent of $i_+$ to $0$}\hfill

\vskip 0.2cm
\noindent
\begin{lemma}
  \label{le1}
  The Markov exponent $M(i_+,0)$ of $i_+$ to $0$ is given by
   \begin{equation}
     \label{M}
     M(i_+,0)=\int_0^{x_+^*}\log f(x)dx,
   \end{equation}where $f(x)$ is defined in \eqref{f} and $i_+=i_+(N)$ is defined in Theorem~\ref{th1}.
\end{lemma}
\begin{proof}
Note that $$\frac{p_i^s}{p_0^s}=\prod_{j=0}^{i-1}\frac{R_0j(1-\delta_1j/N)+R_{1j}(N-j)/N}{(j+1)\Big[1+\frac{\delta_2(j+1)}{N}+\frac{\delta_3\theta}{\theta+(j+1)/N}\Big]}$$
and $$\lim_{N\to\infty}\frac{i_+}{N}=x_+^*.$$  
If $i_+>\lfloor x_+^*N\rfloor,$ then  
\[\begin{split}
  \frac{1}{N}\log\frac{p_{i_+}^s}{p_0^s}=&\frac{1}{N}\log\Big(\frac{R_{10}}{1+\delta_2\frac{1}{N}+\frac{\delta_3\theta}{\theta+\frac{1}{N}}}\Big)+\frac{1}{N}\sum_{j=1}^{\lfloor x_+^*N\rfloor-1}\log \frac{R_0\frac{j}{N}(1-\delta_1\frac{j}{N})+\frac{R_{1j}}{N}(1-\frac{j}{N})}{\frac{j+1}{N}\Big[1+\delta_2\frac{j+1}{N}+\frac{\delta_3\theta}{\theta+\frac{j+1}{N}}\Big]}\\+&\frac{1}{N}\sum_{j=\lfloor x_+^*N\rfloor}^{i_+-1}\log \frac{R_0\frac{j}{N}(1-\delta_1\frac{j}{N})+\frac{R_{1j}}{N}(1-\frac{j}{N})}{\frac{j+1}{N}\Big[1+\delta_2\frac{j+1}{N}+\frac{\delta_3\theta}{\theta+\frac{j+1}{N}}\Big]}.
\end{split}\]  By $\lim_{N\to\infty}\frac{1}{N}\log \alpha_0(N)=0,$ we have \begin{equation}
  \label{5-1}\lim_{N\to\infty}\frac{1}{N}\log\Big(\frac{R_{10}}{1+\delta_2/N+\frac{\delta_3\theta}{\theta+1/N}}\Big)=0.
\end{equation}  By assumption ({\rm H}), as $N\to\infty,$
\begin{equation}
  \label{5-2}\frac{1}{N}\Big|\sum_{j=\lfloor x_+^*N\rfloor}^{i_+-1}\log \frac{R_0\frac{j}{N}(1-\delta_1\frac{j}{N})+\frac{R_{1j}}{N}(1-\frac{j}{N})}{\frac{j+1}{N}\Big[1+\delta_2\frac{j+1}{N}+\frac{\delta_3\theta}{\theta+\frac{j+1}{N}}\Big]}\Big|\leqslant\frac{i_+-\lfloor x_+^*N\rfloor}{N}\log\Big(R_0+1\Big)\to0.
\end{equation}
Note that on one hand,
\[
\begin{split}
  &\frac{1}{N}\sum_{j=1}^{\lfloor x_+^*N\rfloor-1}\log \frac{R_0\frac{j}{N}(1-\delta_1\frac{j}{N})+\frac{R_{1j}}{N}(1-\frac{j}{N})}{\frac{j+1}{N}\Big[1+\delta_2\frac{j+1}{N}+\frac{\delta_3\theta}{\theta+\frac{j+1}{N}}\Big]}\\
  \geqslant&\frac{1}{N}\sum_{j=1}^{\lfloor x_+^*N\rfloor-1}\log \frac{R_0\frac{j}{N}(1-\delta_1\frac{j}{N})}{\frac{j+1}{N}\Big[1+\delta_2\frac{j+1}{N}+\frac{\delta_3\theta}{\theta+\frac{j+1}{N}}\Big]}\\
  =&\frac{1}{N}\Big[\log\Big(R_0\frac{1}{N}(1-\delta_1\frac{1}{N})\Big)-\log\Big(R_0\frac{\lfloor x_+^*N\rfloor}{N}(1-\delta_1\frac{\lfloor x_+^*N\rfloor}{N})\Big)\Big]\\&+\frac{1}{N}\sum_{j=2}^{\lfloor x_+^*N\rfloor}\log \frac{R_0(1-\delta_1\frac{j}{N})}{1+\delta_2\frac{j}{N}+\frac{\delta_3\theta}{\theta+\frac{j}{N}}}.
\end{split}
\]
Since $$\lim_{N\to\infty}\frac{1}{N}\Big[\log\Big(R_0\frac{1}{N}(1-\delta_1\frac{1}{N})\Big)-\log\Big(R_0\frac{\lfloor x_+^*N\rfloor}{N}(1-\delta_1\frac{\lfloor x_+^*N\rfloor}{N})\Big)\Big]=0,$$ we have
\begin{equation}
  \label{5-3}
\begin{split}
  &\liminf_{N\to\infty}\frac{1}{N}\sum_{j=1}^{\lfloor x_+^*N\rfloor-1}\log \frac{R_0\frac{j}{N}(1-\delta_1\frac{j}{N})+\frac{R_{1j}}{N}(1-\frac{j}{N})}{\frac{j+1}{N}\Big[1+\delta_2\frac{j+1}{N}+\frac{\delta_3\theta}{\theta+\frac{j+1}{N}}\Big]}\\
  \geqslant&\liminf_{N\to\infty}\frac{1}{N}\sum_{j=1}^{\lfloor x_+^*N\rfloor-1}\log \frac{R_0\frac{j}{N}(1-\delta_1\frac{j}{N})}{\frac{j+1}{N}\Big[1+\delta_2\frac{j+1}{N}+\frac{\delta_3\theta}{\theta+\frac{j+1}{N}}\Big]}=\int_0^{x^*_+}\log f(x)dx,
\end{split}
\end{equation}
where $f(x)$ is defined in \eqref{f}.

On the other hand, by $\delta_1\leqslant1,$
\[
\begin{split}
  &\frac{1}{N}\sum_{j=1}^{\lfloor x_+^*N\rfloor-1}\log \frac{R_0\frac{j}{N}(1-\delta_1\frac{j}{N})+\frac{R_{1j}}{N}(1-\frac{j}{N})}{\frac{j+1}{N}\Big[1+\delta_2\frac{j+1}{N}+\frac{\delta_3\theta}{\theta+\frac{j+1}{N}}\Big]}\\
  \leqslant&\frac{1}{N}\sum_{j=1}^{\lfloor x_+^*N\rfloor-1}\log \frac{\Big(R_0\frac{j}{N}+\frac{R_{1j}}{N}\Big)(1-\delta_1\frac{j}{N})}{\frac{j+1}{N}\Big[1+\delta_2\frac{j+1}{N}+\frac{\delta_3\theta}{\theta+\frac{j+1}{N}}\Big]}\\
  =&\frac{1}{N}\sum_{j=1}^{\lfloor x_+^*N\rfloor-1}\Big[\log \frac{R_0\frac{j}{N}(1-\delta_1\frac{j}{N})}{\frac{j+1}{N}\Big[1+\delta_2\frac{j+1}{N}+\frac{\delta_3\theta}{\theta+\frac{j+1}{N}}\Big]}+\log\Big(1+\frac{R_{1j}}{R_0j}\Big)\Big]\\
  =&\frac{1}{N}\sum_{j=1}^{\lfloor x_+^*N\rfloor-1}\log \frac{R_0\frac{j}{N}(1-\delta_1\frac{j}{N})}{\frac{j+1}{N}\Big[1+\delta_2\frac{j+1}{N}+\frac{\delta_3\theta}{\theta+\frac{j+1}{N}}\Big]}+\frac{1}{N}\sum_{j=1}^{\lfloor x_+^*N\rfloor-1}\log\Big(1+\frac{R_{1j}}{R_0j}\Big).
\end{split}
\]
By assumption (A1), $R_0\geqslant1,$ which implies $$\frac{1}{N}\sum_{j=1}^{\lfloor x_+^*N\rfloor-1}\log\Big(1+\frac{R_{1j}}{R_0j}\Big)\leqslant\frac{1}{N}\sum_{j=1}^{\lfloor x_+^*N\rfloor-1}\log\Big(1+\frac{1}{j}\Big)=\frac{\log\lfloor x_+^*N\rfloor}{N}\to0,\ \mbox{as}\ N\to\infty.$$
Hence \begin{equation}
  \label{5-4}
\begin{split}
  &\limsup_{N\to\infty}\frac{1}{N}\sum_{j=1}^{\lfloor x_+^*N\rfloor-1}\log \frac{R_0\frac{j}{N}(1-\delta_1\frac{j}{N})+\frac{R_{1j}}{N}(1-\frac{j}{N})}{\frac{j+1}{N}\Big[1+\delta_2\frac{j+1}{N}+\frac{\delta_3\theta}{\theta+\frac{j+1}{N}}\Big]}\\
  \leqslant&\limsup_{N\to\infty}\frac{1}{N}\sum_{j=1}^{\lfloor x_+^*N\rfloor-1}\log \frac{R_0\frac{j}{N}(1-\delta_1\frac{j}{N})}{\frac{j+1}{N}\Big[1+\delta_2\frac{j+1}{N}+\frac{\delta_3\theta}{\theta+\frac{j+1}{N}}\Big]}=\int_0^{x^*_+}\log f(x)dx.
\end{split}
\end{equation} By \eqref{5-3} and \eqref{5-4}, we have \begin{equation}
  \label{5-5}\frac{1}{N}\sum_{j=1}^{\lfloor x_+^*N\rfloor-1}\log \frac{R_0\frac{j}{N}(1-\delta_1\frac{j}{N})+\frac{R_{1j}}{N}(1-\frac{j}{N})}{\frac{j+1}{N}\Big[1+\delta_2\frac{j+1}{N}+\frac{\delta_3\theta}{\theta+\frac{j+1}{N}}\Big]}
=\int_0^{x^*_+}\log f(x)dx.
\end{equation} By \eqref{5-1}, \eqref{5-2} and \eqref{5-5}, we  conclude that \begin{equation}
  \label{lim}\lim_{N\to\infty}\frac{1}{N}\log\frac{p_{i_+}^s}{p_0^s}=\int_0^{x^*_+}\log f(x)dx.
\end{equation}
If $i_+\leqslant\lfloor x_+^*N\rfloor,$ using similar arguments as above, we can still show that  \eqref{lim} holds.
\end{proof}



\end{document}